\documentclass[11pt]{amsart}
\usepackage{amsmath}
\usepackage{amsfonts}
\usepackage{dsfont}
\usepackage{mathrsfs}
\usepackage{amssymb}
\usepackage{amsthm}
\usepackage{enumitem}
\usepackage{bbm}
\usepackage{times}
\usepackage{tabularx}
\usepackage{amsbsy}
\usepackage{mathtools}
\usepackage{tikz}
\usetikzlibrary{arrows,decorations.markings}
\usetikzlibrary{matrix}
\usetikzlibrary{graphs}
\usetikzlibrary{backgrounds}
\usepackage{setspace}     \spacing{1}
\usepackage{color}

\usepackage[colorlinks=true,linkcolor=blue,citecolor=blue]{hyperref}
\usepackage{amssymb,latexsym}
\usepackage{mathrsfs,hyperref}
\usepackage{amsmath,latexsym}
\usepackage{amscd,latexsym} 
\usepackage{graphicx}  
\usepackage[all]{xy}
\usepackage{float}
\usepackage{indentfirst}
\usepackage{amsfonts, latexsym}
\usepackage{graphicx}

\newtheorem{thm}{Theorem}
\newtheorem{prop}{Proposition}
\newtheorem{lem}[prop]{Lemma}
\newtheorem{cor}[prop]{Corollary}

\theoremstyle{definition}

\newtheorem{df}[prop]{Definition} 

\newtheorem{ntn}{Notation}

\newtheorem*{conjs}{Conjecture}
\newtheorem*{qtn}{Question}

\theoremstyle{remark}

\newtheorem{rmk}[prop]{Remark} 
\newtheorem{example}[prop]{Example}

\newcommand{\R}{{\mathbb{R}}}
\newcommand{\Z}{{\mathbb{Z}}}

\newcommand{\Sk}{\mathcal{S}_k}

\newcommand{\M}{\mathcal{M}}

\newcommand{\lsh}{\mathcal{L}\hbox{-shrinking}}

\newcommand{\overbar}{\overline}

\newcommand{\al}{\alpha}

\newcommand{\ga}{\gamma}

\newcommand{\T}{\mathbb{T}}

\newcommand{\nat}{\nabla_t}
\newcommand{\bbN}{\mathbb{N}}
\newcommand{\na}{\nabla}

\DeclareMathOperator{\PS}{(\mathrm{PS})_c}

\def\H{\rm H}

\def\dx{\dot{x}}

\begin{document}

\title[Minimax periodic orbits]{Minimax periodic orbits of convex Lagrangian systems on complete Riemannian manifolds}

\author{Wenmin Gong}

\address{School of Mathematical Sciences, Beijing Normal University, Beijing, 100875, China}

\email{wmgong@bnu.edu.cn}
\date{}

\begin{abstract}
In this paper we study the existence of periodic orbits with prescribed energy levels of convex Lagrangian systems on complete Riemannian manifolds. We extend the existence results of Contreras by developing a modified minimax principle to a class of Lagrangian systems on noncompact Riemannian manifolds, namely the so called $\lsh$ Lagrangian systems. In particular, we prove that for almost every $k\in(0,c_u(L))$ the exact magnetic flow associated to a $\lsh$ Lagrangian has a contractible periodic  orbit with energy $k$. We also discuss the existence and non-existence of closed geodesics on the product Riemannian manifold $\R\times M$. 
\end{abstract}


\maketitle

\section{Introduction}
\noindent Let $M$ be an $m$-dimensional complete Riemannian manifold without boundary. Throughout this paper we shall suppose that $L:TM\to \R$ is a smooth Lagrangian satisfying:
\begin{enumerate}
	\item[{$(C_1)$}]\textsl{Convexity:} for every $(x,v)\in TM$, the second derivative along the fibers $L_{vv}(x,v)$ is uniformly positive definite, namely, in linear coordinates on the fiber $T_xM$, there is a $A_1>0$ such that
	$$w\cdot L_{vv}(x,v)\cdot w\geq A_1|w|^2\quad \forall (x,v)\in TM\;\hbox{and}\;w\in T_xM.$$
	\item[{$(C_2)$}] \textsl{Superlinearity:} For all $R\geq0$ there is $C(R)\geq 0$ such that 
	$$L(x,v)\geq R|v|-C(R).$$
		\item[{$(C_3)$}] \textsl{Boundedness:} $\sup_{x\in M}|L_v(x,0)|<+\infty$ and for every $r>0$,
		$$a(r)=\sup\limits_{\stackrel{(x,v)\in TM,}{|v|<r}}L(x,v)<+\infty,\quad b(r)=\sup\limits_{\stackrel{|w|=1,}{|v|<r}}w\cdot L_{vv}(x,v)\cdot w<+\infty.$$
\end{enumerate}
$(C_3)$  implies that the Hamiltonian associated to $L$ is convex and superlinear. Typical Lagrangians (the so called \textit{electromagnetic Lagrangians}) of the form
$$L(x,v)=\frac{1}{2}|v|^2_x+\theta_x(v)+V(x)$$
satisfy the conditions ($C_1$)--($C_3$) provided that the norm of the one-form $\|\theta\|$ and the function $V$ on $M$ are bounded, where $\|\cdot\|$ is induced by the Riemannian metric on $M$. 

The \textit{Euler-Lagrange flow} $\phi_L^t:TM\to TM$ associated to the Lagrangian $L$ defined as above satisfies $$\phi_L^t(\ga(0),\dot{\ga}(0))=(\ga(t),\dot{\ga}(t)),$$
where $\ga:\R\to M$ is a smooth solution of the Euler-Lagrange equation
\begin{equation}\label{e:lagflow}
\frac{d}{dt}\frac{\partial L}{\partial v}\big(\ga(t),\dot{\ga}(t)\big)-\frac{\partial L}{\partial x}\big(\ga(t),\dot{\ga}(t)\big)=0.
\end{equation} 
The Euler-Lagrange flow is complete, see~\cite{Co, FM}. The \textit{energy function} associated to the above Lagrangian $E_L:TM\to\R$ is defined as 
$$E_L(x,v)=L_v(x,v)[v]-L(x,v).$$
Since the Lagrangian $L$ is time-independent, every solution of (\ref{e:lagflow}) is contained in an energy hypersurface $E_L^{-1}(k)$ for some $k\in\R$. Moreover, the periodic orbits with energy $k$ of the Euler-Lagrange flow correspond to the critical points of the \emph{free period action functional} 
$$\Sk:W^{1,2}(\R/\Z,M)\times (0,\infty)\longrightarrow \R,$$
\begin{eqnarray}
\Sk(x,T)&=&\int^1_0 TL\bigg(x(t),\frac{\dot{x}(t)}{T}\bigg)dt+kT,\notag\\
&=&\int^T_0L(\ga(s),\dot{\ga}(s))ds+kT,\notag
\end{eqnarray}
where $\ga(s)=x(s/T)$. 

In this article we continue the study of the existence of   convex Lagrangians on non-compact Riemannian manifolds that we began in~\cite{Go}. There the author proved that for the flow of an electromagnetic Lagrangian $L$ every energy level $E_L^{-1}(k)$ with $k$ larger than some threshold  contains a non-contractible periodic orbit under some topological and curvature conditions provided that the set of lengths of all closed curves belonging to a given free homotopy class has a positive minimum. 
The critical energy known as \emph{Ma\~{n}\'{e} critical value} defined by
\[
c(L):=\inf\big\{k\in\R\big|\Sk(\ga)\geq0\;\hbox{for every smooth closed  curve $\ga$ on}\; M\big\}.
\]
is very useful since it marks a change in the dynamics and topology~\cite{Ma0,Mat,CIPP1,CIPP2}. 

For our purpose, an important role will be played by two critical values which are defined as
\begin{gather}
e_0(L):=\sup\limits_{x\in M}E_L(x,0),\notag\\
c_u(L):=\inf\big\{k\in\R\big|\Sk(\ga)\geq 0\;\hbox{for every smooth contractible curve $\ga$ on}\; M\notag\big\}.
\end{gather}
Here $c_u(L)$ is called \emph{Ma\~{n}\'{e} critical value of the universal cover}, see~\cite{Co1,CMP}. 

 Since we are mainly interesting in finding periodic solutions of the Euler-Lagrange equation with prescribed energies, without loss of generality we may further assume that $L$ is \textit{quadratic at infinity}, which means that there exists a sufficiently large $R>0$ such that $L(x,v)=|v|_x^2/2$ if $|v|_x>R$. This is because given $k\in\R$ and a Lagrangian $L$ satisfying $(C_1)-(C_3)$, there exists a quadratic at infinity Lagrangian $L_0$ such that $L=L_0$ on $\{E_L\leq k+1\}$, and in this case $c(L)=c(L_0)$, see~\cite{Co1}. Under this additional assumption, for all $(x,v)\in TM$ we have 
\begin{gather}
 A_2|v|^2_x-A_3\leq L(x,v)\leq A_4(1+|v|^2_x) \label{e:quadgrowth1},\\
|L_{v}(x,v)|\leq A_5(1+|v|_x)\label{e:quadgrowth3},
\end{gather}
where $A_i$, $i=2,\ldots,5$ are positive constants. 

On a compact Riemannian manifold of arbitrary dimension Contreras~\cite{Co1} showed that a general uniformly convex Lagrangian $L$ has a contractible periodic orbit with energy $k$ for almost every $k\in (\min E_L,c_u(L))$, and for every $k>c_u(L)$ in the case that $M$ is simply connected. In~\cite{Co1} he also proved that if $M\neq\T^2$ and the energy level $E^{-1}(k)$ is of contact type then $E^{-1}(k)$  contains a periodic orbit.  These existence results were generalized to many other cases, see for instance weakly exact magnetic flow~\cite{Me1},  quadratic at infinity Lagrangians with non-closed $1$-form $\theta_x(v)=L_v(x,0)v$~\cite{Pa} and so on. Except for these results we refer to~\cite{BPV,BPRV,SZ,Go} for recent related results in the case of noncompact manifolds of arbitrary dimensions.


As already observed in~\cite{Go}, periodic orbits of Lagrangian systems on a non-compact manifold may not exist in general if there are no further restrictions on the geometry or topology of a manifold even for the Lagrangian $L(x,v)=|v|^2_x/2$ rising from a Riemannian metric, see also~\cite{Th,BeG,BPV}. The aim of this article is to employ variational methods to extend Contreras's existence results~\cite{Co1} to convex Lagrangian systems on non-compact Riemannian manifolds under suitable assumptions. The main difficulty is caused by the non-compactness of the underlining manifold. Our strategy for overcoming this problem is to use the minimax principle to the action functional $\Sk$ on loops with images in a suitable compact subset of $M$. A new problem occurring when doing this is that the family of sets consisting of loops contained in a fixed compact subset of $M$ is, in general,  not positively invariant under the flow of 
the negative gradient vector field of $\Sk$ in the  usual  sense. This forces us to consider a generalized minimax principle (see Section~\ref{sec:mini}) and introduce the concept of  ``(homotopically) $\lsh$ Lagrangians" in what follows.

\subsection{Main results}

\begin{ntn}
	For any subset $U\subset M$ and any $\varepsilon>0$ we denote by
	$$B_U(\varepsilon)=\big\{x\in M|\hbox{there is a point }y\in U\;\hbox{such that}\; d_M(x,y)<\varepsilon\big\}$$ the $\varepsilon$-neighborhood of $U$ in $M$ and by $\overbar{B}_U(\varepsilon)$ its closure, where the distance function $d_M$ is induced by the Riemannian metric on $M$. 
\end{ntn}

\begin{df}
	Let $M$ be a complete Riemannian manifold.
	We say that a smooth function $L:TM\to\R$ is \textbf{Lagrangian shrinking} (denoted by \textbf{$\lsh$}) if for every bounded set $U\subseteq M$ there is a compact subset $K_0\subseteq M$, a positive number $\varepsilon$ and a smooth map{\footnote{A map defined on an arbitrary subset of a smooth manifold is said to be \textit{smooth} if it can be extended to a smooth map defined on an open neighborhood of this subset.}} $\varphi:K=\overbar{B}_{K_0}(\varepsilon)\to K_0$ such that $U\subseteq K_0$ 
	and $\varphi^*L\leq L$ on $TM|_{K}$, i.e., for all $x\in K$ and all $v\in T_xM$,  
	$L(\varphi(x),d\varphi_x(v))\leq L(x,v)$. Furthermore, whenever $K_0$ is a compact submanifold of $M$ and the composition 
	\begin{equation}\notag
	K_0	\stackrel{\varphi|_{K_0}}{\longrightarrow}  K_0 \stackrel{j}{\hookrightarrow} M
	\end{equation}
	is homotopic to the inclusion $j:K_0\hookrightarrow M$, we say that $L$ is \textbf{homotopically $\lsh$}.
\end{df}

By our convention $B_M(r)=M$ for all $r>0$, so on any closed manifold $M$ every Lagrangian $L:TM\to\R$ is (homotopically) $\lsh$ (by taking $\varphi=id$ and $K_0=M$) . For more examples of  (homotopically) $\lsh$ Lagrangians we refer to Appendix~\ref{app:lsh}.

Our main result is the following two theorems.

\begin{thm}\label{thm: mainresult1}
	Let $M$ be a complete connected Riemannian manifold, and let $L:TM\to\R$ be a Lagrangian satisfying the conditions $(C_1)-(C_3)$.  If $L$ is $\lsh$, then for almost every $\kappa
	\in (\inf E_L, c_u(L))$ the Lagrangian $L$ has a  contractible periodic orbit of positive $(L+\kappa)$-action with energy $\kappa$.
\end{thm}

The key to prove Theorem~\ref{thm: mainresult1} is to employ a modified version of the classical minimax principle together with  Struwe's monotone argument (see~\cite{St2}). 

\begin{conjs}
	Under the assumptions of Theorem~\ref{thm: mainresult1} for every $k\in (\inf E_L,c_u(L))$ the Lagrangian $L$ has a  contractible periodic orbit of positive $(L+k)$-action with energy $k$.
\end{conjs}

Previous existence results of periodic orbits of exact magnetic systems in high energy levels on noncompact manifolds are obtained in~\cite{Go}, but there only the existence of noncontractible periodic orbits was considered. The following theorem complements those results under different assumptions about the manifold and the associated Lagrangians.
\begin{thm}\label{thm: mainresult2}
	Let $M$ be a complete Riemannian manifold and let $L:TM\to\R$ be a Lagrangian satisfying the conditions $(C_1)-(C_3)$. Assume that $L$ is homotopically $\lsh$. If $M$ is simply connected and non-contractible, then for every $k\in (c_u(L),\infty)$ the Lagrangian $L$ has a  periodic orbit of positive $(L+k)$-action with energy $k$; if  $\pi_1(M)\neq0$, then for every non-trivial free homotopy class $\al\in[S^1,M]$ and for every $k
	\in (c_u(L),\infty)$ the Lagrangian $L$ has a  periodic orbit of positive $(L+k)$-action with energy $k$ in the class $\al$.
\end{thm}

\begin{rmk}
Those periodic orbits with energy levels belonging to  $(e_0(L),\infty)$ given in Theorem~\ref{thm: mainresult1} and Theorem~\ref{thm: mainresult2} are not constant closed curves. Because any constant closed curve $x_0$ has energy $E_L(x_0,0)\leq e_0(L)$ but there each periodic orbit has energy $k>e_0(L)$.  
\end{rmk}

\subsection{Closed orbits of exact magnetic flows}
Let $\theta$ be a $1$-form on a Riemannian manifold $M$. 
Recall that an \textit{exact magnetic flow} is the flow of the Lagrangian
\begin{equation}\label{e:magneticL}
L(x,v)=\frac{1}{2}|v|_x^2+\theta_x(v)
\end{equation}
which, in dimension two case,  is a model for the motion of a particle under the effect of a magnetic field, see~\cite{Ar,No,Gi2,Ta,CMP}. It is obvious that $e_0(L)=0$ because $E_L(x,v)=|v|_x^2/2$. It was proved by Contreras \textit{et al}~\cite{CMP} that any exact magnetic flow on a closed surface possesses periodic orbits in all energy levels. Prior to this result Taimanov~\cite{Ta,Ta1} proved that for every $k\in(e_0(L),c_u(L))$ there exists a local minimizer of the free period action functional $\Sk$ for the electromagnetic Lagrangian~(\ref{e:magneticL}) over the space of absolutely continuous periodic curves on a closed surface. 
For more recent results about the existence of periodic orbits of an exact magnetic flow on closed surfaces we refer to~\cite{GiG,Sc2,Sc3,Go}.  

A direct corollary of our main result is the following. 

\begin{cor}\label{cor: mainresult3}
Let $M$ be a complete connected Riemannian manifold and let $L:TM\to\R$ be as in (\ref{e:magneticL}). Assume that  $L$ is $\lsh$ with bounded $\|\theta\|$. Then for almost every $k\in(0,c_u(L))$ the energy level $E_L^{-1}(k)$ contains a contractible periodic orbit. If in addition $M$ is simply connected and non-contractible and $L$ is homotopically $\lsh$, then for every $k
\in (c_u(L),\infty)$  the Lagrangian $L$ has a periodic orbit with energy $k$; if  $\pi_1(M)\neq0$, then for every non-trivial free homotopy class $\al\in[S^1,M]$ and for every $k
\in (c_u(L),\infty)$ the Lagrangian $L$ has a  periodic orbit of positive $(L+k)$-action with energy $k$ in the class $\al$.
\end{cor}

Typical examples which satisfy the assumptions of Corollary~\ref{cor: mainresult3}  can be found in Example~\ref{eg:ls}. 

\begin{qtn}
	When does an exact magnetic flow possess a contractible periodic orbit with energy $k$ for every $k\in(0,c_u(L))$ on a non-compact Riemannian manifold? 
\end{qtn}

Let us recall that the magnetic flow of a pair $(g,\sigma)$ on $TM$ is dual to the Hamiltonian flow of $H(x,p)=\|p\|^2/2$ on $T^*M$ of $M$ with respect to the twisted symplectic form
$$\omega_\sigma=\omega_0+\pi^*\sigma$$
via the metric $g$ on $M$~\cite{CMP,Gi2,GiK}, where $\sigma$ is a closed $2$-form on $M$, $\pi:T^*M\to M$ is the natural projection and $\omega_0$ is the standard symplectic form on $T^*M$. So in order to answer the above question one might appeal to the methods from symplectic geometry, for instance, Floer theory.  Indeed, for any symplectic form $\sigma$ on a closed manifold $M$, Usher~\cite{Us} established the existence of contractible magnetic geodesics for every low energy level using Floer theory based on previous work of Ginzburg and G\"{u}rel~\cite{GiG}. For more existence results of periodic orbits of magnetic flows which are builded upon on Floer homology groups defined on the cotangent bundle $T^*M$,  we refer to~\cite{FS,GiG1,Go1,GoX,Me2,Sch}. But here it is unclear to the author how to define  Floer homology on the cotangent bundle $T^*M$ of a non-compact manifold $M$. We leave this question for further research.

\subsection{Existence and non-exitence of closed geodesics on $\R\times M$}

Let $M$ be a closed Riemannian manifold with a metric $g$. Consider the product Riemannian manifold $N=\R\times M$ with the metric
$$h(r,x)=dr^2+\beta(r,x)g(x)\quad\forall (r,x)\in N,$$
where $\beta:N\to\R$ is a smooth positive function.

\begin{thm}[\textbf{Benci and Giannoni}] If $\beta$ is independent of the variable $x\in M$ and $\beta'\neq 0$ on $\R$, then $N$ does not possess a nonconstant closed geodesic. 
\end{thm}
For a proof of this theorem we refer to ~\cite[Prop. 2.2]{BeG}. 

Contrary to this non-existence result, if the positive function $\beta$ takes a different form then closed geodesics on $N$ may occur. Indeed, we have the following.

\begin{thm}
	Let $\lambda_{\pm}$ be two smooth positive monotone increasing functions defined on $(0,\infty)$. If $\beta(r,x)=\lambda_+(r)$ on $[R_+,\infty)\times M$ and  $\beta(r,x)=\lambda_-(-r)$ on $(-\infty, -R_-]\times M$  for two positive numbers $R_{\pm}$, then $N$ has a nonconstant closed geodesic. 
\end{thm}

\begin{proof}

For the Lagrangian $L(z,v)=h_z[v,v]/2$ we have that $e_0(L)=c_u(L)=c(L)=0$. It is not hard to show that $L$ is homotopically $\lsh$, for a similar proof this fact we refer to Appendix~\ref{app:lsh}. 
So it follows from Corollary~\ref{cor: mainresult3} that $N$ has a nonconstant closed geodesic.
\end{proof}

\section{The free period action functional and Hilbert manifolds}\label{subsec:funcst}
	Let $(M,g)$ be a $m$-dimensional complete Riemannian manifold without boundary. Set $\T=\R/\Z$. Consider the Sobolev space of loops
	$$W^{1,2}(\T,M)=\bigg\{x:\T\to M\big|x\;\hbox{is absolutely continuous and} \int_{\T}|x'(t)|_{x(t)}^2dt<\infty\bigg\},$$
	where $|\cdot|$ is the norm induced by the metric $g$.
	The tangent space of $W^{1,2}(\T,M)$ at $x$ is naturally identified with the space of $W^{1,2}$ sections of $x^*(TM)$. $W^{1,2}(\T,M)$ is a smooth Hilbert manifold equipped with the Riemannian metric as follows
	\begin{equation}\notag
	\langle \xi,\eta\rangle_x:=\langle \xi(0), \eta(0)\rangle_{x(0)}+\int_{\T}\langle \nat\xi(t), \nat\eta(t)\rangle_{x(t)} dt\quad \forall \xi, \eta\in T_xW^{1,2}(\T,M),
	\end{equation}
	where $\nat$ denotes the Levi-Civita covariant derivative along $x$. 
	Since $M$ is complete, $W^{1,2}(\T,M)$ is a complete Hilbert manifold, see~\cite{Kl,Pal0}. 
	
	Denote $\M=W^{1,2}(\T,M)\times (0,\infty)$ and define  its metric as 
	\begin{equation}\label{e:prodmetr}
	\big\langle(\xi,\al),(\eta,\beta)\big\rangle_{(x,T)}=\al\beta+\langle \xi,\eta\rangle_x.
	\end{equation}
Then $\M$ is a smooth Hilbert manifold but not complete, see~\cite{Co1,Go}. Each point $(x,T)\in\M$ corresponds to the $T$-periodic curve $\ga(t)=x(t/T)$.  In the following we denote by $\|\cdot\|$ the norm and by $d_\M(\cdot,\cdot)$ the distance function with respect to the metric~(\ref{e:prodmetr}). For a compact subset $K\subseteq M$ we shall denote by
$$\M(K):=\big\{(x,T)\in\M\big|x(\T)\subseteq K\}.$$
the subset of $\M$. 
   
Given $k\in\R$, the \emph{free period action functional} $\Sk:\M\longrightarrow \R$ is defined by
$$
\Sk(x,T)=\int^1_0 TL\bigg(x(t),\frac{\dot{x}(t)}{T}\bigg)dt+kT\quad \forall (x,T)\in\M.
$$
It is easy to check that the functional $\Sk$ is in $C^{1,1}(\M)$ and is twice G\^{a}teaux differentiable at every point. 

For any $(\xi,\al)\in T_{(x,T)}\M$, the differential of the action functional $\Sk$ at $(x,T)$ is computed as follows: 
\begin{eqnarray}\label{e:diffSk}
d\Sk(x,T)[\xi,\al]&=&\int^1_0\big\{T L_x\big(x(t),\dx(t)/T\big)[\xi(t)]+L_v\big(x(t),\dx(t)/T\big)[\dot{\xi}(t)]\big\}dt\notag\\
&&+\al\int^1_0\big\{k-E_L\big(x(t),\dx(t)/T\big)\big\}dt
\end{eqnarray}
from which we see that $(x,T)$ is a critical point of $\Sk$ if and only if $\ga(t)=x(t/T)$ is a periodic orbit of the Euler-Lagrangian equation (\ref{e:lagflow}) with $E_L(\ga,\dot{\ga})=k$.

\begin{rmk}\label{e:timediffSk}
	$$\frac{\partial\Sk}{\partial T}(x,T)[\al]=\al\int^1_0\big\{k-E_L\big(x(t),\dx(t)/T\big)\big\}dt=\frac{\al}{T}\int^T_0\big\{k-E_L\big(\ga(s),\dot{\ga}(s)\big)\big\}ds,$$
	where $\ga(s)=x(s/T)$. 
\end{rmk}

\begin{lem}\label{lem:bddbelow}
	If $k\geq c_u(L)$ then $\Sk$ is bounded from below on each connected component of $\M$; If $k<c_u(L)$  then $\Sk$ is unbounded from below on each connected component of $\M$.
\end{lem}

The proof of the above lemma can be found in~\cite[Lemma~4.1]{Co1} or~\cite[Lemma~8]{Go}.


We end up this section with the following lemma. 
\begin{lem}\label{lem:boundDist}
	Let $p(s)$, $s\in [0,1]$ be a path in $\M$ connecting $(x_0,T_0)$ to 
	$(x_1,T_1)$. If $\|p'(s)\|\leq \delta$ for all $s\in[0,1]$ then $d_M(x_0(\T),x_1(\T))\leq\sqrt{1+2\sqrt{6}}\delta$ and $|T_1-T_0|\leq \delta$.
\end{lem}

\begin{proof}
   Write $p(s)=(x_s(t),T_s)$. By definition for all $s\in[0,1]$ we have
	\begin{equation}\label{e:pathspeed}
	\|p'(s)\|^2=\bigg|\frac{dT_s}{ds}\bigg|^2+\bigg|\frac{\partial x_s(0)}{\partial s}\bigg|^2_{x_s(0)}+\int^1_0\big|\nabla_s\dot{x}_s(t)\big|^2_{x_s(t)}dt\leq \delta^2
	\end{equation}
	from which we get 
	\[
	|T_1-T_0|=\bigg|\int^1_0\frac{dT_s}{ds}ds\bigg|\leq
	\int^1_0\bigg|\frac{dT_s}{ds}\bigg|ds\leq\delta.
	\]
	 
   Set
   $$G(t):=\int^1_0\bigg|\frac{\partial x_s(t)}{\partial s}\bigg|^2_{x_s(t)}ds.$$
   
   To estimate $G(t)$ we first note that $G(0)\leq\delta^2$ by (\ref{e:pathspeed}), then we find that
   \begin{eqnarray}   
   G(t)-G(0)&=&\int^t_0\frac{d}{d\tau}G(\tau)d\tau\notag\\
   &=&\int^t_0\int^1_0\frac{d}{d\tau}\bigg|\frac{\partial x_s(\tau)}{\partial s}\bigg|^2_{x_s(\tau)}dsd\tau\notag\\
   &=&2\int^t_0\int^1_0\bigg\langle\na_\tau\frac{\partial x_s(\tau)}{\partial s},\frac{\partial x_s(\tau)}{\partial s} \bigg\rangle_{x_s(\tau)}  dsd\tau\notag\\
   &=&2\int^t_0\int^1_0\bigg\langle\na_s\dx_s(\tau),\frac{\partial x_s(\tau)}{\partial s} \bigg\rangle_{x_s(\tau)}  dsd\tau.\notag
   \end{eqnarray}
  Then by the H\"{o}lder inequality and (\ref{e:pathspeed}) we have
	\begin{eqnarray}   
	\big|G(t)-G(0)\big|&\leq&2 \bigg[\int^1_0\int^1_0\big|\na_s\dx_s(\tau)\big|^2_{x_s(\tau)} dsd\tau\bigg]^\frac{1}{2} \bigg[\int^t_0\int^1_0\bigg|\frac{\partial x_s(\tau)}{\partial s}\bigg|^2_{x_s(\tau)} dsd\tau\bigg]^\frac{1}{2} \notag\\
	&\leq &2\delta \bigg(\int^t_0G(\tau)d\tau\bigg)^\frac{1}{2}.\notag
   \end{eqnarray}
	Hence
	\begin{equation}\label{e:funG}
	G(t)\leq G(0)+2\delta \bigg(\int^t_0G(\tau)d\tau\bigg)^\frac{1}{2}\leq \delta^2+2\delta \bigg(\int^t_0G(\tau)d\tau\bigg)^\frac{1}{2}. 
	\end{equation}
	Let
	$$f(t)=\bigg(\int^t_0G(\tau)d\tau\bigg)^\frac{1}{2}.$$
	Then by (\ref{e:funG}) we get
	$$\frac{d}{dt}f(t)^2\leq \delta^2+2\delta f(t).$$
	To further estimate $f(t)$ we borrow the trick from the proof of~\cite[Lemma 2.3]{Co1} as follows. Set
	$$\mu_0:=\sup \big\{\mu\in[0,1]\big|f(t)\leq 2\delta(t+\frac{1}{2}),\quad \forall t\in[0,\mu]\big\}.$$
	Then for all $s\in[0,\mu_0]$ it holds that
	$$\frac{d}{dt}f(t)^2\leq \delta^2+2\delta\cdot 2\delta(t+\frac{1}{2})\leq 4\delta^2(t+1),$$
	hence
	$$f(t)^2\leq 4\delta^2\bigg(\frac{t^2}{2}+t\bigg),$$
	$$f(t)\leq 2\delta \sqrt{\frac{t^2}{2}+t}.$$
	
	Note that $\sqrt{t^2/2+t}< t+1/2$ for all $t\geq 0$. This implies $\mu_0=1$. Therefore by  (\ref{e:funG}) for all $t\in[0,1]$ we have
	$$G(t)\leq \delta^2+2\delta f(t)\leq \big(1+2\sqrt{6}\big)\delta^2.$$
	Consequently, 
	$$d_M(x_0(t),x_1(t))\leq \int^1_0\bigg|\frac{\partial x_s(t)}{\partial s}\bigg|_{x_s(t)}ds\leq \bigg(\int^1_0\bigg|\frac{\partial x_s(t)}{\partial s}\bigg|^2_{x_s(t)}ds\bigg)^\frac{1}{2}\leq \sqrt{1+2\sqrt{6}}\delta$$
	which implies the lemma.

\end{proof}

\section{The modified minimax principle}\label{sec:mini}

In this section we shall give a variant version of the general minimax principle developed by Palais~\cite{Pal} (or see~\cite[Theorem~4.2]{St1}).

Let $X$ be a $C^{\infty}$ Hilbert manifold and let $f\in C^{1,1}(X,\R)$. We say that a flow $\phi_t:X\to X$ is \textsl{positively complete} if for every $x\in X$ the map $t\mapsto \phi_t(x)$ is  defined on $[0,\infty)$. Let $\mathcal{F}$  be a family of subsets of $M$. If there is another family $\mathcal{F}'$ of subsets of $M$ and a map $\Phi:\mathcal{F}'\to \mathcal{F}$ such that $\phi_t(F)\in \mathcal{F}'$ for all $F\in\mathcal{F}$ and all $t\geq 0$ and that $\Phi(F')\subseteq \{x\in X|f(x)<a\}$ whenever $F'\subseteq \{x\in X|f(x)<a\}$ for all $F'\in \mathcal{F}'$ and all $a\in \R$, we say that $\mathcal{F}$ is \textsl{positively $(\phi,\Phi)$-invariant}. 

We remark here that whenever $\mathcal{F}=\mathcal{F}'$ the property that $\mathcal{F}$ is positively $(\phi,Id)$-invariant is equivalent to say that $\mathcal{F}$ is \textsl{positively $\phi$-invariant} in the usual sense (see~\cite{St1, Pal}). 

If $f$ is a $C^{1,1} $ function on a complete Hilbert manifold $X$, one can not guarantee that the negative gradient flow of $f$ is positively complete since in general  $-\na f$ is only a locally Lipschitz vector field. But if  we rescale $-\na f$ appropriately, for instance, putting $$V_f=-\frac{\na f}{\sqrt{1+\|\na f\|^2}},$$ then the flow of $V_f$ is positively complete. In fact, suppose that $\ga:[0,a)\to X$ is a solution of the equation
$$\dot{\ga}(t)=V_f(\ga(t))$$
with $\ga(0)=p\in X$. If $a$ is finite, due to $\|V_f\|<1$ the length of $\ga$ is finite, hence the set $\ga([0,a))$ is contained in a compact set of $X$ because of the completeness of $X$. So this solution of the above equation can be extended further if $a$ is finite.

Let $X$ be a $C^{\infty}$ Hilbert manifold and let $f\in C^{1,1}(X,\R)$. Recall that a sequence $\{x_n\}$ in $X$ is said to be a \textit{Palais–Smale sequence at level $c$} (simply denoted by $\PS$) if $\lim_{n\to\infty}f(x_n)=c$ and $\lim_{n\to\infty}\|df(x_n)\|=0$. Here $\|\cdot\|$ is the norm induced by the Riemannian metric on $X$. The function $f$ is said to satisfy the \textit{Palais–Smale condition at level $c$} ($\PS$ condition for short) if every $\PS$ sequence has a convergent subsequence. 

\begin{thm}[\textbf{The modified minimax principle}]\label{thm:minimaxP}
	Let $X$ be a complete $C^{\infty}$ Hilbert manifold and let $f\in C^{1,1}(X,\R)$. Let $\phi_t$ be  the flow of the vector field $V_f$ defined as above. Suppose that there is a family $\mathcal{F}$ of subsets of $M$ which is positively $(\phi,\Phi)$-invariant. If the number
	$$c=c(f,\mathcal{F}):=\inf\limits_{F\in \mathcal{F}}\sup\limits_{x\in F} f(x)$$
	is finite, then $f$ has a $\PS$ sequence. If in addition $f$ satisfies the $\PS$ condition, then $c$ is a critical value.  
\end{thm}
\begin{proof}
	By contradiction, we assume that there is $\varepsilon>0$ such that $\|df\|\geq \varepsilon$ on
	 $\{x\in X|c-\varepsilon\leq f(x)\leq c+\varepsilon\}$.
	 By definition, there is a subset $F\in\mathcal{F}$ such that 
	 $$f(x)\leq c+\varepsilon\quad \forall x\in F.$$
	 Let $x\in F$. If $x\in F\cap \{x\in X|f(x)\geq c-\varepsilon\}$, we find
	 \begin{eqnarray}
	 f(x)-f(\phi_t(x))&=&-\int^t_0\frac{d}{dt}f(\phi_t(x))dt\notag\\
	 &=&-\int^t_0 df(\phi_t(x))[V_f(\phi_t(x))]dt\notag\\
	 &=&\int^t_0\frac{\|\na f(\phi_t(x))\|^2}{\sqrt{1+\|\na f(\phi_t(x))\|^2}}dt,\label{e:decreasef1}
	 \end{eqnarray}
	 hence $f$ is non-increasing along the flow line of $\phi_t$. If $\phi([0,T]\times\{x\})\subseteq\{x\in X||f-c|\leq\varepsilon\}$ then
	 by (\ref{e:decreasef1}) we have
	 $$f(\phi_T(x))\leq f(x)-\frac{\varepsilon^2 T}{\sqrt{1+\varepsilon^2}}\leq c+\varepsilon-\frac{\varepsilon^2 T}{\sqrt{1+\varepsilon^2}}.$$
	 So if $T>2\sqrt{1+\varepsilon^2}/\varepsilon$ then 
	 \begin{equation}\label{e:decresef}
	 f(\phi_T(x))< f(x)-\varepsilon.
	 \end{equation}

	 If $x\in F\cap \{x\in X|f(x)< c-\varepsilon\}$, since $f$ is non-increasing along the flow of $V_f$ one still have (\ref{e:decresef}) for any $T\geq 0$. Therefore by taking $T>2\sqrt{1+\varepsilon^2}/\varepsilon$ we conclude that 
	 $$\phi_T(F)\subseteq \big\{x\in X|f(x)<c-\varepsilon\big\}.$$
	 Since $\mathcal{F}$ is positively $(\phi,\Phi)$-invariant, we have $\phi_T(F)\in \mathcal{F}'$
	 and 
	 $$\mathcal{F}\ni\Phi(\phi_T(F))\subseteq \big\{x\in X|f(x)<c-\varepsilon\big\}$$
	 which contradicts the definition of the minimax value $c$. This completes the proof of the theorem. 
\end{proof}

\begin{rmk}
	In order to apply the generalized minimax principle, it is often useful to rescale the gradient vector field $\na f$ in various forms to obtain the desired positively complete flows. For example, a truncated flow fixing some sublevel of $f$ can be constructed as follows: let $\chi:\R\to\R$ be a smooth bounded function such that $\chi$ vanishes on $(-\infty, \lambda]$ and satisfies $\chi>0$ on $(\lambda, \infty)$, then the subset $\{x\in X|f(x)\leq\lambda\}$ is invariant under the flow of the vector field $$V^\chi_f=\frac{\chi(f)\na f}{1+\|\na f\|^2}.$$
	Clearly, the flow of $V^\chi_f$ is positively complete. Now if we assume that $\mathcal{F}$  is positively $(\phi,\Phi)$-invariant with this truncated flow and that $\lambda<c(f,\mathcal{F})$, then through a similar argument in the proof of Theorem~\ref{thm:minimaxP} one can find a $\PS$ sequence of $f$. 
\end{rmk}

\begin{rmk}
Although Theorem~\ref{thm:minimaxP} is not used directly in this article, the basic idea behinds this generalized minimax principle is the key to prove our main result. In certain cases we believe that when using the minimax principle to geometric problems the condition that the flow $\phi_t$ of a negative gradient vector field (or pseudo-gradient vector field) is positively $\phi$-invariant is unnecessary and could be loosen by using the geometric or topological properties of the underlying manifold itself. 
\end{rmk}

\section{The Palais-Smale condition}

Let $K\subseteq M$ be a compact set. We have the following Palais-Smale condition.

\begin{prop}\label{prop:PSc}
If $\{(x_n,T_n)\}_{n\in\bbN}\subset\M(K)$ is a $\PS$ sequence for $\Sk$ with $0<D_1\leq T_n\leq D_2<\infty$,  then this sequence has a convergent subsequence. 	
\end{prop}
\begin{proof}
	By our assumption, $T_n$ is bounded, so up to a subsequence we may assume that $\{T_n\}$ converges to $T_*>0$.  It suffices to prove that $\{x_n\}$ has a convergent subsequence in $W^{1,2}$. 
	In what follows we will embed $M$ isometrically in $\R^N$ (for $N$ large enough) which is equipped with the Euclidean metric by the Nash's embedding theorem.
	
	Since $\{(x_n,T_n)\}_{n\in\bbN}\subset\M(K)$ is a $\PS$ sequence for $\Sk$, by~(\ref{e:quadgrowth1}) we have
	\begin{eqnarray}
	c+o(1)=\Sk(x_n,T_n)&=&T_n\int^1_0 L\bigg(x_n(t),\frac{\dot{x}_n(t)}{T_n}\bigg)dt+kT_n\notag\\
	&\geq&\frac{A_2}{T_n}\int^1_0|\dx_n(t)|^2dt+\big(A_3-k\big)T_n\notag\\
	&\geq&\frac{A_2}{D_2}\int^1_0|\dx_n(t)|^2dt-|A_3-k|D_2\label{e:bdfromT}
	\end{eqnarray}
from which we get that $\|\dx_n\|_{L^2}$ is bounded uniformly, where $\|\cdot\|_{L^2}$ denotes the $L^2$-norm with respect to the Riemannian metric on $M$. 
So $\{x_n\}$ is a bounded sequence in $W^{1,2}(\T, M)$ due to $x_n(\T)\subseteq K$ for all $n\in \bbN$. Up to considering a subsequence, we assume that there is $x\in W^{1,2}(\T, M)$ such that
\begin{gather}
x_n\stackrel{n}{\longrightarrow} x \quad\hbox{weakly in}\; W^{1,2}\; \hbox{and strongly in}\;L^\infty.\label{lim:wk}
\end{gather}

In the following we will prove that $\{x_n\}$ converges strongly to $x$ in $W^{1,2}$. Denote $$\pi(z):\R^N\to T_zM\quad\forall z\in M$$ the orthogonal projection onto $T_zM$, and consider
$$w_n(t)=\pi(x_n(t))[x_n(t)-x(t)].$$
Since $\pi(z)$ is smooth with respect to $z\in M$ and the image of every $x_n$ belongs to the compact set $K$, $w_n$ is bounded in $T_{x_n}W^{1,2}$. Since $\{x_n\}$ is a $\PS$ sequence, by (\ref{e:diffSk}) taking $(\xi, \al)=(w_n,0)$ implies
\begin{equation}\label{e:converge1}
\int^1_0\big\{T_nL_{x}(x_n(t),\dx_n(t)/T_n)[w_n]+L_v(x_n(t),\dx_n(t)/T_n)[\dot{w}_n]\big\}dt=o(1).
\end{equation}
If $x\in K$ then $|L_x(x,v)|\leq C(K)|v|^2$ for some constant $C(K)>0$. Hence,
\begin{equation}\notag
\big|L_{x}(x_n(t),\dx_n(t)/T_n)[w_n(t)]\big|\leq C(K)\bigg(\frac{|\dx_n|^2_{x_n}}{T_n^2}+1\bigg)|w_n(t)|_{x_n(t)}.
\end{equation}
Since $\|\dx_n\|_{L^2}$ is bounded uniformly and $w_n$ uniformly converges to zero, we conclude that the first integral in (\ref{e:converge1}) is infinitesimal. So we have
\begin{equation}\label{e:converge2}
\int^1_0L_v(x_n(t),\dx_n(t)/T_n)[\dot{w}_n(t)]dt=o(1).
\end{equation}
Set
\begin{equation}\label{e:un}
u_n(t)=\pi^{\perp}(x_n(t))[x_n(t)-x(t)]
\end{equation}
where $\pi^{\perp}(z)={\rm id}-\pi(z):\R^N\to T_zM$, then we get
\begin{equation}\label{e:decomxn}
x_n-x=w_n+u_n.
\end{equation}
 
The convexity hypothesis $(C_1)$ about the Lagrangian implies that
\begin{eqnarray}
&&L_v(x_n(t),\dx_n(t)/T_n)[\dx_n(t)-\dx(t)]-L_v(x_n(t),\dx(t)/T_n)[\dx_n(t)-\dx(t)]\notag\\
&=&\frac{1}{T_n}\int^1_0L_{vv}\bigg(x_n(t),\frac{\dx_n(t)}{T_n}+s\frac{\dx_n(t)-\dx(t)}{T_n}\bigg)\big[\dx_n(t)-\dx(t),\dx_n(t)-\dx(t)\big]ds\notag\\
&\geq&\frac{A_1}{T_n}|\dx_n(t)-\dx(t)|^2\label{e:estLv}.
\end{eqnarray}
By (\ref{e:quadgrowth3}), $L_v(x_n(t),\dx(t)/T_n)$ converges strongly in $L^2$. So the fact that $\dx_n(t)-\dx(t)$ converges weakly to zero in $L^2$ implies that
\begin{equation}\label{e:converge3}
\int^1_0L_v(x_n(t),\dx(t)/T_n)[\dx_n(t)-\dx(t)]dt=o(1).
\end{equation}
Combining (\ref{e:converge2})--(\ref{e:converge3}) yields
\begin{eqnarray}\notag
o(1)+\int^1_0L_v(x_n(t),\dx_n(t)/T_n)[\dot{u}_n]dt
\geq \frac{A_1}{T_n}\int^1_0|\dx_n(t)-\dx(t)|^2 dt.
\end{eqnarray}
By (\ref{e:quadgrowth3}), $L_v(x_n(t),\dx_n(t)/T_n)$ is uniformly bounded in $L^2$ due to the boundedness of $\|\dx_n\|_{L^2}$. Therefore to prove that $\{x_n\}$ converges strongly to $x$ in $W^{1,2}$, it suffices to prove that $\dot{u}_n$ converges strongly to $0$ in $L^2$.
	It follows from (\ref{e:un}) that
	\begin{equation}\label{e:du_n}
	\dot{u}_n(t)=\pi^\perp\big(x_n(t)\big)\big[\dot{x}_n(t)-\dot{x}(t)\big]+d\pi^\perp\big(x_n(t)\big)\big[\dot{x}_n(t)\big]\big[x_n(t)-x(t)\big].
	\end{equation}
	Using (\ref{lim:wk}) and the assumption that $x_n(\T)\subset K$ for all $n$ we deduce that 
	\begin{equation}\label{e:1thofdu_n}
	d\pi^\perp\big(x_n(t)\big)\big[\dot{x}_n(t)\big]\big[x_n(t)-x(t)\big]\stackrel{n}{\longrightarrow} 0 \quad\hbox{strongly in}\; L^2
	\end{equation}
	and
	 \begin{equation}\label{e:2thofdu_n}
	\pi^\perp\big(x_n(t)\big)\big[\dot{x}_n(t)-\dot{x}(t)\big]=-\pi^\perp\big(x_n(t)\big) \big[\dot{x}(t)\big]\stackrel{n}{\longrightarrow}0, \quad\hbox{strongly in}\; L^\infty,
	\end{equation}
	where we have used $\pi^\perp\big(x(t)\big) \big[\dot{x}(t)\big]=0$ because $\dot{x}(t)\in T_{x(t)}M$ for all $t\in\T$.
	By (\ref{e:du_n})--(\ref{e:2thofdu_n}) we have that 
	$\|\dot{u}_n\|_{L^2}$ tends to $0$. This completes the proof. 
	
\end{proof}

The following lemma, which was observed first by Contreras~\cite{Co1}, asserts that the only place where a $\PS$ sequence $(x_n,T_n)$ with $T_n\to 0$ fails to have a convergent subsequence is the level zero. 
\begin{lem}\label{lem:timeBL}
	Suppose that $L$ is quadratic at infinity.
	Let $\{(x_n,T_n)\}_{n\in\bbN}$ be a $\PS$ sequence of $\Sk$ which satisfies $T_n\to 0$ as $n\to\infty$. Then $c=0$.
\end{lem}
\begin{proof}
	By the definition of $\PS$ sequence and (\ref{e:diffSk}),
	we have that
	\begin{equation}\label{e:PSc}
	o(1)=\frac{\partial\Sk}{\partial T}(x_n,T_n)=\int^1_0\big\{k-E_L\big(x_n(t),\dx_n(t)/T_n\big)\big\}dt.
	\end{equation}
	The assumption that $L$ is quadratic at infinity and the boundedness $(C_3)$ of $L$ imply
	\begin{gather}
	|L(x,v)|\leq a_0(|v|^2_x+1)\label{e:upperL}\\
	E_L(x,v)\geq a_1|v|_x^2-a_2\label{e:lowerE}
	\end{gather}
	for some positive constants $a_1,a_2$ and $a_3$.
	It follows from (\ref{e:PSc})  and (\ref{e:lowerE}) that
	\begin{equation}\label{e:estimatec0}
	k+o(1)\geq \int^1_0\bigg(a_1\frac{|\dx_n(t)|^2}{T_n^2}-a_2\bigg)dt.
	\end{equation}
	Since $\{(x_n,T_n)\}$ is a $\PS$ sequence, by (\ref{e:upperL}) we have 
	\begin{equation}\label{e:estimatec1}
	|c|+o(1)=|\Sk(x,T)|\leq T_n\int^1_0a_0\bigg(\frac{|\dx_n(t)|^2}{T_n^2}+1\bigg)dt+kT_n.
	\end{equation}
	Combining (\ref{e:estimatec0}) and (\ref{e:estimatec1}) implies $c=0$ provided that $T_n$ tends to zero. This completes the proof. 
	
\end{proof}

\section{The geometry of $\Sk$ on the space of loops in a compact set}
\subsection{The mountain pass geometry}\label{sec:mountain}
In this section we show that the action functional $\Sk$ on the space of loops supported in a compact set has the structure of ``mountain pass geometry". The proof of such a fact, which is essentially due to Contreras~\cite{Co1} in the case that $M$ is compact, is based on a locally quadratic isoperimetric inequality presented in the following. 

\begin{lem}[\cite{Co1}]\label{lem:isop}
	Let $\theta$ be a smooth $1$-form on $M$. Let $x_0\in M$ and let $U\subset M$ be an open subset centered at $x_0$ whose closure is diffeomorphic to a closed ball in $\R^m$. Then there exists a number $\mu>0$ such that if $\ga$ is a closed curve in $U$ then
	\begin{equation}\label{e:quadIsoIneq}
	\bigg|\int^1_0\ga^*\theta\bigg|\leq\mu \ell(\ga)^2,
	\end{equation}
	where $\ell(\ga)$ denotes the length of any curve $\ga$ in $M$.  
\end{lem}
For the proof of this lemma we refer to \cite[Lemma~5.1]{Co1}.


\begin{prop}\label{prop:mountain}
	Let $K$ be a compact set in $M$ and let $k>e_0(L)$. 
	Then there is a positive number $a=a(K)$ such that 
	if $p:[0,1]\to\M(K)$ is a continuous path connecting any constant loop $(x_0,T)$ with $x_0\in K$ to a loop in $K$ satisfying $\Sk(p(1))<0$, then 
	$$\sup\limits_{s\in[0,1]}\Sk(p(s))>a.$$
\end{prop}

\begin{proof}
	
	Write $p(s)=(x_s,T_s)$, and consider the $1$-form on $M$ defined by
	\begin{equation}\label{e:oneform}
	\theta(x)[v]:=L_v(x,0)[v]\quad \forall v\in T_xM.
	\end{equation}
	
	Since $K$ is compact in $M$, there are finitely many open subsets $U_j$, $j=1,\ldots, n$ which cover $K$ such that for every $j$ the inequality (\ref{e:quadIsoIneq}) holds for the $1$-form given by (\ref{e:oneform}) with some number  $\mu_j(K)>0$ provided that $\ga$ is in $U_j$.
	Let $\widetilde{U}_j=U_j\cap K$ for all $j$, and let $\delta$ be a Lebesgue number for this covering of $K$. 
	
	Set $\mu=\inf\mu_j(K)$. In the following we will show that if $\Sk(x_1,T_1)<0$ then
	\begin{equation}\label{e:estlength}
	\ell(x_1)\geq\min\bigg\{\delta,\frac{\sqrt{A_1(k-e_0(L))}}{\sqrt{2}\mu}\bigg\}=:d.
	\end{equation}
	
	Clearly, if $\ell(x_1)\geq \delta$ then (\ref{e:estlength}) holds. If $\ell(x_1)< \delta$ then by the definition of the Lebesgue number there exists some $j\in\{1,2,\ldots,n\}$ such that $x_1$ is contained in $\widetilde{U}_j$ and thus in $U_j$.
	
	From our assumptions about $L$ we deduce that 
	\begin{eqnarray}
	L(x,v)&=&L(x,0)+L_v(x,0)[v]+\frac{1}{2}L_{vv}(x,sv)[v,v]\notag\\
	&\geq&\frac{A_1}{2}|v|^2_x+\theta(x)[v]-E_L(x,0),\label{e:lowBdL}
	\end{eqnarray}
	where the number $s=s(x,v)\in[0,1]$ in the first equality is given by the Taylor expansion. Then by~(\ref{e:lowBdL}) we have
	\begin{eqnarray}
		0>\Sk(x_1,T_1)&\geq& T_1\int^1_0\bigg\{A_1\frac{|\dx_1(t)|^2}{2T_1^2}+\frac{1}{T_1}\theta(x_1(t))[\dx_1(t)]\bigg\}dt\notag\\&&+T_1\int^1_0\big(k-E_L(x_1(t),0)\big)dt\notag\\
		&\geq&\frac{A_1}{2T_1}\ell^2(x_1)-\mu\ell^2(x_1)+T_1\big(k-e_0(L)\big)\label{e:SkLowBd}
    \end{eqnarray}
which, together with $k>e_0(L)$, implies that $T_1>\frac{A_1}{2\mu}$ and that
	$$\ell^2(x_1)>\frac{T_1(k-e_0(L))}{\mu}>\frac{A_1(k-e_0(L))}{2\mu^2}.$$
	So in this case we still have (\ref{e:estlength}). 
	
	Note that the length of every constant loop is zero. Fixing $r\in (0,d]$, by the continuousness of $\ell({x_s})$ with respect to $s\in[0,1]$ there exists some $s_0\in[0,1]$ such that $\ell({x_{s_0}})=r$.
	
   By (\ref{e:SkLowBd}) we find that
   \begin{eqnarray}\Sk(p(s_0))&\geq& \sqrt{2A_1(k-e_0(L))}\ell({x_{s_0}})-\mu\ell^2(x_{s_0})\notag\\
   &=&\sqrt{2A_1(k-e_0(L))}r-\mu r^2=:a.
   \end{eqnarray}
	Then $a>0$ because for $r\in[0,d]$ it holds that
	$$r\leq\frac{\sqrt{A_1(k-e_0(L))}}{\sqrt{2}\mu}
	<\frac{\sqrt{2A_1(k-e_0(L))}}{\mu}.$$
	
\end{proof}

Although Proposition~\ref{prop:mountain}
does not hold for $k<e_0(L)$, we have a variant of this proposition for the energy interval $(\inf E_L,e_0(L))$. If  $k>\inf E_L$, by definition there exists a point $x_*$  such that  $k>E_L(x_*,0)$.
Let $N$ be any compact topological submanifold (with boundary) of codimension $0$ in $M$ containing the point $x_*$ as an inner point. Following Taimanov~\cite{Ta0,Ta2}, we consider the set of continuous maps
\[ 
\begin{split}
\mathcal{D}_{N}=\big\{(z,\tau): &[0,1]\times N\to\M(N)\big|\;z(0,x_0)\equiv x_0\\&\hbox{and}\;\Sk(z(1,x_0),\tau(1,x_0))<0\;\forall x_0\in N\big\}.
\end{split}
\]
It turns out that $\mathcal{D}_{N}$ is not empty. This is a well-known fact which can be verified by using Bangert’s technique of  ``pulling
one loop at a time" see~\cite{Ba} for  this argument. Moreover, we have the 	analogue of Proposition~\ref{prop:mountain}:

\begin{prop}\label{prop:mountain2}
	For $k\in(\inf E_L, e_0(L))$, let $\mathcal{D}_{N}$ be the set given as above. 
	Then there is a positive number $a=a(N)$ such that 
    for every $(z,\tau)\in\mathcal{D}_{N}$ we have that 
	$$\sup\limits_{(s,x_0)\in[0,1]\times N}\Sk(z(s,x_0),\tau(s,x_0))>a.$$
\end{prop}

The proof of this proposition is similar to that of Proposition~\ref{prop:mountain}, for completeness we shall give a sketch of it. 

Given $k\in(\inf E_L, e_0(L))$, we pick a neighbourhood $V\subset N$ of $x_*$ so that
$$\sup_{x\in V}E_L(x,0)=-\inf_{x\in V}L(x,0)=:\sigma<k.$$

Let $U\subseteq V$ be an open ball centered at $x_*$ given by Lemma~\ref{lem:isop}. Let $A_1,\mu$ be as before. 
Set 
$$0<r<d:=\min\bigg\{\frac{\hbox{diam}(U)}{2},\frac{\sqrt{A_1(k-\sigma)}}{\sqrt{2}\mu}\bigg\}.$$

In what follows we shall show that for any $(z,\tau)\in\mathcal{D}_{N}$ the path $(z(s,x_*),\tau(s,x_*))$ of loops (denote it by $\Gamma(s)$ for short) has a time $0<s_0<1$ such that the loop $\Gamma(s_0)$ contained in $U$ has length $\ell(\Gamma(s_0))$ equal to $r$. 

Since $\Gamma(0)=x_*$, by the continuousness of $\Gamma(s)$ with respect to $s\in[0,1]$ there exists $s_1\in(0,1]$ such that  $\Gamma(s)\subset U$ for all $s\in[0,s_1]$. Letting $(x_1,T_1)=\Gamma(s_1)$,  similar to (\ref{e:SkLowBd}) if $\Sk(x_1,T_1)<0$ we have
\begin{equation}\label{e:SkLowBd'}
0>\Sk(x_1,T_1)\geq\frac{A_1}{2T_1}\ell^2(x_1)-\mu\ell^2(x_1)+T_1\big(k-\sigma\big).
\end{equation}
This, due to $k>\sigma$, implies that $\frac{A_1}{2T_1}-\mu<0$, and thus $T_1>\frac{A_1}{2\mu}$.
Then from (\ref{e:SkLowBd'}) we deduce that 
$$\ell^2(x_1)>\frac{T_1(k-\sigma)}{\mu-\frac{A_1}{2T_1}}>\frac{T_1(k-\sigma)}{\mu}>\frac{A_1(k-\sigma)}{2\mu^2}\geq d^2.$$
By the continousness of $\ell(\Gamma(s))$ with respect to $s$, there exists $0<s_0<1$ such that $\ell(\Gamma(s_0))=r$.    

Finally, similar to (\ref{e:SkLowBd'}), setting $(x_0,T_0)=\Gamma(s_0)$ we find that
\begin{eqnarray}
\Sk(\Gamma(s_0))&\geq&
\frac{A_1}{2T_0}\ell^2(\Gamma(s_0))-\mu\ell^2(\Gamma(s_0))+T_0\big(k-\sigma\big)\notag\\
&\geq& \sqrt{2A_1(k-\sigma)}\ell(\Gamma(s_0))-\mu\ell^2(\Gamma(s_0))\notag\\
&=&\sqrt{2A_1(k-\sigma)}r-\mu r^2=:a(N)>0.
\end{eqnarray}
Since $[0,1]\times\{x_*\}\subseteq [0,1]\times N$, we conclude the desired result.

\subsection{The simply-connected case } \label{subsec:simplyC}
Suppose that $M$ is simply connected and non-contractible. In this case we have that $c(L)=c_u(L)$. 
Let $K$ be a compact submanifold of $M$ such that $\pi_{n+1}(K)\neq0$ for some $n>0$. Using the tubular neighborhood theorem one can see that the inclusions
$$C^\infty(\T,K) \hookrightarrow\M(K)\hookrightarrow C^0(\T,K) $$
are homotopy equivalences. Here we remark that it is unclear to us that whether these inclusions are homotopy equivalences provided that $K$ is only a compact set (the tubular neighborhood theorem is not applicable). This the right reason that we ask that $K_0$ is a compact submanifold of $M$ in the definition of  ``homotopically $\lsh$".

Due to the isomorphism~$\pi_n(C^0(\T,K))\cong\pi_{n+1}(K)$, there exists a homotopically non-trivial free class
$$0\neq\mathcal{H}\in \big[S^n,C^0(\T,K))\big]\cong \big[S^n,\M(K)\big]$$
Consider the family of subsets in $\M$
$$\mathcal{G}_K=\big\{g(S^n)\big|g\in C^0(S^n,\M(K))\;\hbox{and}\;[g]=\mathcal{H}\big\}.$$

\begin{prop}\label{prop:simplyC}
	If $k>c_u(L)$, then
	$$d(k):=\inf\limits_{g(S^n)\in\mathcal{G}_K}\sup\limits_{(x,T)\in g(S^n)}\Sk(x,T)>0.$$
\end{prop}
	
\begin{proof}
Since $\mathcal{H}$ is non-trivial and $K$ is compact, by~\cite[Theorem~2.1.8]{Kl} we have that
$$l:=\inf\limits_{g(S^n)\in\mathcal{G}_K}\sup\limits_{(x,T)\in g(S^n)}\ell(x)>0.$$
Here $\ell(\ga)$ denotes the length of any curve $\ga$ in $M$ with respect to some Riemannian metric of $M$.
If $(x,T)\in\M$ is a loop with $\ell(x)\geq l$ and $\Sk(x,T)\leq A$ for some number $A>0$, then by (\ref{e:quadgrowth1}) we have that
\begin{equation}\label{e:above}
\begin{split}
A\geq \Sk(x,T)\geq& A_2\int^T_0|\dx(t)|^2dt-A_3T+kT\\
\geq& A_2\frac{\ell(x)^2}{T}-(A_3-k)T,
\end{split}
\end{equation}
where we have required that $A_3\gg k$. Due to $\ell(x)\geq l$ we have that 
$$AT+(A_3-k)T^2-A_2l^2\geq 0.$$
It follows that
$$T\geq\frac{-A+\sqrt{A^2+4A_2(A_3-k)l^2}}{2(A_3-k)}=:\mathcal{T}>0.$$
So we have that
$$\Sk(x,T)=\mathcal{S}_{c(L)}(x,T)+(k-c(L))T\geq (k-c(L))\mathcal{T}>0$$
for $k>c(L)=c_u(L)$. By definition, there exists $g(S^n)\in\mathcal{G}_K$ such that 
$$\sup\limits_{(x,T)\in g(S^n)}\Sk(x,T)\leq d(k)+1.$$
Let $A=d(k)+1$. Then the above argument shows that the minimax value $d(k)$ is strictly positive.
\end{proof}

	
\section{The properties of $\lsh$ Lagrangians  }\label{sec:lshProperty}
Let $M$ be a complete Riemannian manifold and let $L:TM\to\R$ be a $\lsh$ Lagrangian. Suppose that $U$ is bounded subset of $M$. By definition, there exists a compact set $K_0$, a positive number $\varepsilon$ and a smooth map $$\varphi:K:=\overbar{B}_{K_0}(\varepsilon )\to K_0$$
such that $U\subseteq K_0$ and $\varphi^*L\leq L$ on $TM|_{K}$. 


Consider the map 
\[\Phi_K:\M(K)\longrightarrow \M(K_0),\quad \Phi_K(x,T)=(\varphi(x),T),
\]
where $\varphi(x)(t):=\varphi(x(t))$. 
Since $\varphi^*L\leq L$ on $TM|_{K}$, for any $c\in\R$ we have
\begin{equation}\label{e:lshproperty}
\Phi_K\big(\M(K)\cap\big\{\Sk\leq c\big\}\big)\subseteq
\M(K_0)\cap\big\{\Sk\leq c\big\}.
\end{equation}

Moreover, whenever $L$ is homotopically $\lsh$, one can define the map
\begin{gather}
\Psi:C^0(S^n,\M(K))\longrightarrow C^0(S^n,\M(K_0)),\quad g\longmapsto  g_0\notag
\end{gather}
which satisfies for all $n\in\bbN$
$$g_0(z)=\Phi_K(g(z))\quad \forall z\in S^n,\;\forall\;g\in C^0(S^n,\M(K))$$
By definition,  if we restrict $\Psi$ to the set $C^0(S^n,\M(K_0))$ then it induces the identical map on the set $[S^n,\M(K_0)]$ consisting of free homotopy classes.


\section{The proof of Theorem~\ref{thm: mainresult1}}
In this section we will apply a modified minimax principle to show the existence of contractible periodic orbits of convex Lagrangian systems with  energy levels in $(\inf E_L,c_u(L))$.

If $e_0(L)<k<c_u(L)$, by definition there exists a contractible loop $\ga(t)=x(t/T)$ such that
$$\Sk(x,T)<0.$$
Since the contractible component of the free loop space $\Lambda M$ is connected, there is a continuous path $p:[0,1]\to\M$ connecting a constant loop to $\ga$. Denote the union of the images of each loop $p(s)$ by $W$, i.e.,
$$W=\bigcup\limits_{s\in[0,1]}x_s(\T),$$
where $p(s)=(x_s,T_s)\in\M$. It is clear that $W$ is a compact subset in $M$. Then, due to the fact that $L$ is $\lsh$, by definition there exists a compact set $K_0$, a positive number $\varepsilon$ and a smooth map $$\varphi:K:=\overbar{B}_{K_0}(\varepsilon )\to K_0$$
such that $W\subseteq K_0$ and $\varphi^*L\leq L$ on $TM|_{K}$. 

Consider a family of subsets in $\M$
\[ 
\begin{split}
\mathcal{F}_{K}=\big\{p([0,1])\big|\;&p\in C^0([0,1],\M(K)),\;p(0)\;\hbox{is a constant loop in}\;K\;\hbox{and}\\&p(1)\;\hbox{is a loop in}\;K\;\hbox{satisfying}\;\Sk(p(1))<0\big\}.
\end{split}
\]
In a similar vein we define $\mathcal{F}_{K_0}$ via replacing $K$ by $K_0$. By our choice of $K$ and $K_0$ it holds that
$\emptyset\neq\mathcal{F}_{K_0}\subseteq\mathcal{F}_{K}$. It follows from (\ref{e:lshproperty}) that
$\varphi$ induces a map
\begin{equation}\label{e:pushback map}
\Phi_K:\mathcal{F}_K\to \mathcal{F}_{K_0}. 
\end{equation}

Now we consider the minimax value 
\begin{equation}\label{e:minimaxVal}
c(k)=\inf\limits_{F\in\mathcal{F}_{K_0}}\sup\limits_{(x,T)\in F}\Sk(x,T).
\end{equation}
By virtue of Proposition~\ref{prop:mountain}, if $k\in(e_0(L),c_u(L))$ then $c(k)>0$.

 Set $I=(e_0(L),c_u(L))$. Since for each loop $\ga$ the function $k\mapsto\Sk(\ga)$ is non-decreasing, the function $k\mapsto c(k)$ is non-decreasing. By Lebesgue's theorem there is a full measure subset in $I$ which is given by
 \[
 J=\big\{k\in I\big|\exists D>0, \;\exists\epsilon_0>0\; \hbox{such that}\;|c(k+\epsilon)-c(k)|<D|\epsilon|, \;\forall |\epsilon|<\epsilon_0\big\}.
 \]
 
In the following we will use a modified minimax principle and Struwe's arguments~\cite{St2} to show the following lemma.
\begin{lem}\label{lem:PS}
If $\kappa\in J$, then $\mathcal{S}_\kappa$ has a $\mathrm{(PS)}_{c(\kappa)}$ sequence of contractible loops in $K$ with uniformly bounded periods. 
\end{lem}

\begin{proof}[The proof of Lemma~\ref{lem:PS}] 
	Let $\{k_n\}_{n\in\bbN}$ be a sequence with $k_n> \kappa$ and $\lim_{n\to\infty}k_n=\kappa$. Set $\eta_n=k_n-\kappa<1$.  Then (\ref{e:minimaxVal}) implies that for every $n$ there exists $F_n\in\mathcal{F}_{K_0}$ such that 
	$$\max\limits_{(x,T)\in F_n}\mathcal{S}_{k_n}\leq c(k_n)+\eta_n.$$
	
	If $(x,T)\in F_n\cap\big\{(x,T)\in\M|\mathcal{S}_{\kappa}(x,T)\geq c(\kappa)-\eta_n\big\}$, then for sufficiently large $n$ due to $\kappa\in J$ we have that
	\[
	T=\frac{\mathcal{S}_{k_n}(x,T)-\mathcal{S}_{\kappa}
		(x,T)}{k_n-
		\kappa}\leq\frac{c(k_n)+\eta_n-c(\kappa)+\eta_n}
	{\eta_n}\leq D+2
	\]
	for some $D>0$, and that
	\[
	\mathcal{S}_{\kappa}(x,T)\leq\mathcal{S}_{k_n}(x,T)\leq c(k_n)+\eta_n\leq c(\kappa)+(D+1)\eta_n.
	\]
	Denote
	$$\mathcal{B}_n=\big\{(x,T)\in\M(K_0)\big|\;T\leq D+2\;\hbox{and}\;\mathcal{S}_\kappa(x,T)\leq c(\kappa)+(D+1)\eta_n\big\}.$$
	The above argument shows that
	\begin{equation}\label{e:bdT&A}
	F_n\subseteq \mathcal{B}_n\cup\big\{(x,T)\in\M(K_0)\big|\mathcal{S}_{\kappa}(x,T)< c(\kappa)-\eta_n\big\}.
	\end{equation}

	Let $\tau>0$ be a parameter (to be determined later) and let $\rho:\R\to[0,\tau]$ be a smooth function such that
	\[ 
	\rho(t)= 
	\begin{cases}\tau,\quad  \ \ & \mbox{if}\;
	t\in \big[\frac{c(\kappa)}{2},\infty),\\ 
	0, \quad \ \ & \mbox{if}\;t\in(-\infty,\frac{c(\kappa)}{4}).
	\end{cases} 
	\] 
	
	Set
	\[
	X=-\frac{\rho(\mathcal{S}_\kappa)\nabla \mathcal{S}_\kappa}{\sqrt{1+\|\nabla \mathcal{S}_\kappa\|^2}}.
	\]
	Let $\phi_t$ be the flow of $X$ on $\M$. Although 
	$\M$ is not complete, $\phi_t$ is positively complete because the sublevel set $\{(x,T)\in\M|\mathcal{S}_{\kappa}(x,T)\leq c(\kappa)/4\}$ is invariant under this flow and $X$ is bounded, and the restriction of $\M$ to $\{T\geq T_*\}$ is complete for every $T_*>0$. 
	
	Moreover, by Lemma~\ref{lem:boundDist}, the fact that
	$\|X\|\leq \tau$ implies that for every $F\in\mathcal{F}_{K_0}$ the images of all loops in $\phi_1(F)$ are contained in $\overbar{B}_{K_0}(C\tau)$,
	where $C=\sqrt{1+2\sqrt{6}}$.

	Fix $0<\tau<\varepsilon/C$. Since $B_{K_0}(\varepsilon)\subset K$ and $\mathcal{S}_{\kappa}$ is non-increasing along the flow line of $\phi_t$, by (\ref{e:bdT&A}) and Lemma~\ref{lem:boundDist} we have that 
	\begin{equation}\label{e:flowaway}
	\phi_1(\mathcal{F}_{K_0})\subseteq \mathcal{F}_{K} 
	\end{equation}
	and that
	\begin{equation}\label{e:bdTA}
	\phi([0,1]\times F_n)\subseteq \mathcal{Q}_n\cup\big\{(x,T)\in\M(K)\big|\mathcal{S}_{\kappa}(x,T)< c(\kappa)-\eta_n\big\},
	\end{equation}
	where $\mathcal{Q}_n:=\big\{(x,T)\in\M(K)\big|\;T\leq D+2+\tau\;\hbox{and}\;\mathcal{S}_\kappa(x,T)\leq c(\kappa)+(D+1)\eta_n\big\}$.
	
	In the following we shall show that there exists a sequence of  contractible loops $(x_n,T_n)\in\M(K)$ such that
	\begin{gather}
	(x_n,T_n)\in\phi([0,1]\times F_n)\cap\{\mathcal{S}_{\kappa}\geq c(\kappa)-\eta_n\},\notag\\
	\|d\mathcal{S}_{\kappa}(x_n,T_n)\|\stackrel{n}{\longrightarrow} 0.\notag
	\end{gather}
	By (\ref{e:bdTA}) such a sequence of loops in $K$ has uniformly bounded periods $T_n$.
	
Arguing by contradiction, we assume that there exists $0<\delta<1$ such that
\[
\|d\mathcal{S}_{\kappa}\|\geq\delta\quad\hbox{on}\;\phi([0,1]\times F_n)\cap\{\mathcal{S}_{\kappa}\geq c(\kappa)-\eta_n\}
\]
for sufficiently large $n$. We further require that $c(\kappa)-\eta_n\geq c(\kappa)/2>0$.
	
For any $(x,T)\in F_n$, if $\phi([0,1]\times\{(x,T)\})\subseteq \{\mathcal{S}_{\kappa}\geq c(\kappa)-\eta_n\}$ then we have
\begin{eqnarray}
\mathcal{S}_{\kappa}(\phi_1(x,T))&=&\mathcal{S}_{\kappa}(x,T)+\int^1_0\frac{d}{ds}\mathcal{S}_{\kappa}(\phi_s(x,T))ds\notag\\
&=&\mathcal{S}_{\kappa}(x,T)-\int^1_0\frac{\rho(\mathcal{S}_\kappa)\|\nabla \mathcal{S}_\kappa\|^2}{\sqrt{1+\|\nabla \mathcal{S}_\kappa\|^2}}(\phi_s(x,T))ds\notag\\
&\leq&c(\kappa)+(D+1)\eta_n-\frac{\tau\delta^2}
{\sqrt{1+\delta^2}},\notag
\end{eqnarray}
consequently, for $n$ large enough we have
\begin{equation}\label{e:sublevel}
\mathcal{S}_{\kappa}(\phi_1(x,T))\leq c(\kappa)-\eta_n.
\end{equation}
	
Since $\mathcal{S}_{\kappa}\circ \phi_t$ is non-increasing with $t$, for $(x,T)\in F_n\cap\{\mathcal{S}_{\kappa}\leq c(\kappa)-\eta_n\}$ we still have (\ref{e:sublevel}). 
So we have that
$$\max\limits_{(x,T)\in\phi_1(F_n)}\mathcal{S}_{\kappa}(x,T)\leq c(\kappa)-\eta_n.$$
Then it follows from (\ref{e:lshproperty}) and (\ref{e:flowaway}) that
$$\max\limits_{(x,T)\in\Phi_K(\phi_1(F_n))}\mathcal{S}_{\kappa}(x,T)\leq c(\kappa)-\eta_n$$
which obviously contradicts the definition of $c(\kappa)$ because $\Phi_K(\phi_1(F_n))\subset\mathcal{F}_{K_0}$ by (\ref{e:pushback map}) and (\ref{e:flowaway}). This completes the proof of the claim.
	
\end{proof}

\begin{proof}[The proof of Theorem~\ref{thm: mainresult1}] 
We first consider the case that $\kappa\in (e_0(L),c_u(L))$. 
	By Lemma~\ref{lem:PS}, for almost every $\kappa\in (e_0(L),c_u(L))$ there exists a Palais-Smale sequence of loops $(x_n,T_n)$ in some compact subset $K\subseteq M$ for $\mathcal{S}_{\kappa}$ with uniformly bounded $T_n$ at level $c(\kappa)$. Since $c(\kappa)>0$ by Proposition~\ref{prop:mountain}, it follows from Lemma~\ref{lem:timeBL} that $T_n$ is bounded away from zero. So there are two numbers $D_1$ and $D_2$ such that $$0<D_1\leq T_n\leq D_2<\infty.$$
By Proposition~\ref{prop:PSc}, the sequence  $(x_n,T_n)$ has a limiting point which gives us the periodic orbit of positive $(L+\kappa)$-action with energy $\kappa\in (e_0(L),c_u(L))$.
 
For the case that $\kappa\in(\inf E_L,e_0(L))$, the proof is similar to that in the above case. We only point out that  Proposition~\ref{prop:mountain2} plays the same role as Proposition~\ref{prop:mountain} with $\mathcal{F}_{K}$ replaced by $\mathcal{D}_{N}$ as given in Section~\ref{sec:mountain}, where $N$ is, in addition, required to satisfy the $\lsh$ property, i.e., there exists a compact subset $N_0$ of $M$, a positive number $\varepsilon$ and a smooth map $$\varphi:N=\overbar{B}_{N_0}(\varepsilon )\to N_0$$
such that $x_*\in N_0$ and $\varphi^*L\leq L$ on $TM|_{N}$.

\end{proof}

\section{The proof of Theorem~\ref{thm: mainresult2}}
In this section we prove the existence of periodic orbits of convex Lagrangian systems with  energy levels in $(c_u(L),\infty)$.

\begin{proof}[The proof of Theorem~\ref{thm: mainresult2}]
\textbf{We prove the first part of the theorem.} Since 
$M$ is non-contractible, there exists a homotopically non-trivial map $f\in C^0(S^{n+1},M)$ for some $n\geq0$, see~\cite[p.405, Cor.24]{Sp}. Since $f(S^{n+1})$ is a compact set in $M$ and $L$ is homotopically $\lsh$, there exists a compact submanifold $K_0$, a positive number $\varepsilon$ and a smooth map 
$$\varphi:K:=\overbar{B}_{K_0}(\varepsilon )\to K_0$$
such that $f(S^{n+1})\subseteq K_0$ and $\varphi^*L\leq L$ on $TM|_{K}$, and that the composition of the restriction of $\varphi$ to $K_0$ with the inclusion $K_0\hookrightarrow M$ is homotopic to this inclusion. Clearly, $\pi_{n+1}(K_0)\neq 0$ because $f$ is homotopically trivial in $M$ if it is in $K_0$. So there exists a homotopically non-trivial free class
$$0\neq\mathcal{H}\in \big[S^n,C^0(\T,K_0))\big]\cong \big[S^n,\M(K_0)\big].$$

Consider the collection of subsets in $\M$
$$\mathcal{G}_{K_0}=\big\{g(S^n)\big|g\in C^0(S^n,\M(K_0))\;\hbox{and}\;[g]=\mathcal{H}\big\}.$$
It follows from Proposition~\ref{prop:simplyC} that for all $k>c_u(L)$ we that
	$$d(k):=\inf\limits_{g(S^n)\in\mathcal{G}_{K_0}}\sup\limits_{(x,T)\in g(S^n)}\Sk(x,T)>0.$$

Next we show
\begin{lem}\label{lem:PSc}
	For every $k\in(c_u(L),\infty)$,  $\mathcal{S}_k$ has a $\mathrm{(PS)}_{d(k)}$ sequence of contractible loops in $K$ with uniformly bounded periods. 
\end{lem}
The proof of Lemma~\ref{lem:PSc} is similar to that of Lemma~\ref{lem:PS}, so we will only give a sketch of it. The difference between these two lemmata is that we can obtain a Palais-Smale sequence  of contractible loops  with uniformly bounded periods for \textbf{all} energy levels in $(c_u(L),+\infty)$ in Lemma~\ref{lem:PSc} but only obtain one for \textbf{almost all} energy levels in $(\inf E_L,c_u(L))$ in Lemma~\ref{lem:PS} due to the usage of Struwe's argument. 

\noindent{\bf The proof of Lemma~\ref{lem:PSc}.}
Firstly, by definition, for every $k\in(c_u(L),\infty)$ one can find a sequence of subsets $g_l(S^n)\in\mathcal{G}_{K_0}$, $l\in\bbN$ such that 
\[\Sk(x,T)\leq d(k)+\epsilon_l,\quad\forall (x,T)\in g_l(S^n),\]
where $\epsilon_l>0$ is infinitesimal. As before, one can rescale the negative gradient vector field $\nabla \Sk$ suitably to a positively complete and bounded vector field on $\M$ so that $\phi\big([0,1]\times g_l(S^n)\big)\subset \M(K)$ by virtue of Lemma~\ref{lem:boundDist}, where $\phi(t,\cdot)$ denotes the flow of this rescaled vector field. Moreover, we have that
\[
\phi_1\big(g_l(S^n)\big)\subseteq \mathcal{P}_n\cup\big\{(x,T)\in\M(K_0)\big|\mathcal{S}_{k}(x,T)< d(k)-\epsilon_l\big\}.
\]
Here $\mathcal{P}_n$ is a subset of $\M$ whose elements have bounded $\mathcal{S}_{k}$-actions. We notice that elements of  $\mathcal{P}_n$ have uniformly bounded periods. In fact, we have the identity
\[
\Sk(x,T)=\mathcal{S}_{c_u(L)}(x,T)+(k-c_u(L))T,
\]
and $\mathcal{S}_{c_u}$ is bounded from below by Lemma~\ref{lem:bddbelow}. So for every $k>c_u(L)$ the periods $T$ of elements $(x,T)$ with bounded $\Sk$-actions are bounded from above.  

Next, we show that there exists a sequence of  contractible loops 
\[
(x_l,T_l)\in\phi([0,1]\times g_l(S^n))\cap\{(x,T)\in\M(K)\big|\mathcal{S}_{k}\geq d(k)-\epsilon_l\}
\]
with uniformly bounded $T_l$ such that $\|d\mathcal{S}_{k}(x_l,T_l)\|\stackrel{l}{\longrightarrow} 0$. To do this, by contradicion we argue as before and deduce that 
\[
\max\limits_{(x,T)\in \phi_1(g_l(S^n))}\Sk(x,T)\leq d(k)-\epsilon_l
\]
for sufficiently large $l$. Fix such an integer $l$. The properties of homotopically  $\lsh$ Lagrangian show that
\begin{equation}\label{e:loweraction}
\max\limits_{(x,T)\in \Phi_K\circ\phi_1(g_l(S^n))}\Sk(x,T)\leq d(k)-\epsilon_l.
\end{equation}
Observe that $\Phi_K\circ\phi_1\circ g_l$ represents $\mathcal{H}$. This is because in the space $C^0(S^n,\M(K_0))$ $\Phi_K\circ\phi_1\circ g_l$ is homotopic to $\Phi_K\circ g_l$ via $\Phi_K\circ\phi_t\circ g_l$, $t\in[0,1]$ and $\Phi_K\circ g_l$ is homotopic to $g_l$  (see Section~\ref{sec:lshProperty}). So
(\ref{e:loweraction}) is in contradiction with the definition of the minimax value $d(k)$. 

Finally, we proceed the proof exactly as in the proof of Theorem~\ref{thm: mainresult1} by replacing Proposition~\ref{prop:mountain} with Proposition~\ref{prop:simplyC}. This completes the proof of the first part of  Theorem~\ref{thm: mainresult2}.

\textbf{Now we prove the second part of the theorem.} The proof is similar to the one in the first part. The key is to use a  modified minimax principle to obtain a Palais-Smale sequence of loops taking values in a compact subset of $M$.

Let $\ga$ be any 
non-contractible closed curve which represents the class $\al$. Since $L$ is homotopically $\lsh$, one can take a compact subset $K_1$ containing the image of $\ga$ in $M$, a positive number $\varepsilon$ and a smooth map 
$$\varphi:K:=\overbar{B}_{K_1}(\varepsilon )\to K_1$$
such that $\varphi^*L\leq L$ on $TM|_{K}$ and that the composition of the restriction of $\varphi$ to $K_1$ with the inclusion $K_1\hookrightarrow M$ is homotopic to this inclusion. Let $\M_\al(K_1)$ (resp. $\M_\al(K)$) denote the connected component of $\M(K_1)$ (resp. $\M(K)$) corresponding to $\al$. By Lemma~\ref{lem:bddbelow}, for $k>c_u(L)$ the functional $\Sk$ is bounded from below on each non-trivial free homotopy class. So  the number
$$e(k):=\inf\big\{\Sk(x,T)|(x,T)\in\M_\al(K_1)\big\}$$
is finite for every $k\in(c_u(L),\infty)$.  By definition, there exists a sequence of points $(x_n,T_n)\in\M_\al(K_1)$ such that 
\[\Sk(x_n,T_n)\leq e(k)+\eta_n,\quad\forall  n\in\bbN\]
where $\eta_n>0$ is infinitesimal. Once again, we can rescale the negative gradient vector field $\nabla \Sk$ suitably to show that there exists a $(PS)_{e(k)}$ sequence $\{(y_n,\tau_n)\}$ of  bounded $\Sk$-actions  such that 
\[
(y_n,\tau_n)\in\phi\big([0,1]\times \{(y_n,\tau_n)\}\big)\cap\{(x,T)\in\M_\al(K)\big|\mathcal{S}_{k}\geq e(k)-\eta_n\},
\]
where $\phi(t,\cdot)$ is the flow of the corresponding rescaled vector field. 
As before,  since $k>c_u(L)$ by Lemma~\ref{lem:bddbelow}  the periods $\tau_n$ are bounded from above. Moreover, 
$\tau_n$ is bounded away from zero. In fact, due to the  compactness of $K$ the infimum of lengths of loops with images in $K$ representing the free homotopy class $\al$ is positive.  Since $\Sk(y_n,\tau_n)$ are bounded, from (\ref{e:above}) we deduce that if $\tau_n\to 0$  as $n\to\infty$ then the lengths of loops $\ga_n(t)=y_n(t/\tau_n)$ go to zero, which is impossible since each $\ga_n$ represents $\al$. Then it follows from Proposition~\ref{prop:PSc} that we have a convergent subsequence of $\{(y_n,\tau_n)\}$  and hence a periodic orbit representing $\al$. This completes the proof.

\end{proof}

\section{Appendix: Examples of homotopically $\lsh$ Lagrangians}\label{app:lsh}

\begin{example}
	Let $(M,g)$ be a noncompact complete Riemannian manifold. Suppose that there exists an increasing sequence of compact submanifolds $M_i\subset M$, that is,
	$$M_1\subset M_2\subset M_i\subset \cdots \subset M\quad \hbox{and}\quad \bigcup\limits_{i}M_i=M$$
	and a vector field $V$ on $M$ which  is everywhere transverse to each $\partial M_i$ pointing outwards and satisfies that 
	$$\mathbb{L}_Vg=\lambda g$$
	for some nonnegative number $\lambda$, where $\mathbb{L}_V$ is the Lie derivative along $V$. Then the standard kinetic energy $L(x,v)=\|v\|^2_g/2$ on $TM$ is homotopically $\lsh$. 
	
\end{example}

\begin{example}\label{eg:ls}
	Let $M$ be a closed Riemannian manifold with the metric $g$. Let $\lambda,\mu\in C^\infty([0,\infty),\R)$ be two positive and monotone increasing functions and let $\phi$ be a positive function on $M$. Consider $\widetilde{M}=\R^k\times M$, where $k\geq 1$. Define the metric on $\widetilde{M}$ as
	$$\widetilde{g}(x,y)=\phi(y)g_{\R^k}(x)+\lambda(|x|^2)g(y)
	\quad \forall (x,y)\in\widetilde{M},$$
	where $g_{\R^k}$ and $|\cdot|$ are the standard metric and the norm with respect to
	the standard inner product on $\R^k$ respectively. Let $\theta$ be a smooth $1$-form on $M$. Let $V$ be a smooth function on $M$ and let $\pi: \widetilde{M}\to M$ be the projection map. Let $L:T\widetilde{M}\to\R$ be a Lagrangian of the form
	$$L(z,v)=\frac{1}{2}\widetilde{g}_z[v,v]+\big(\pi^*\theta\big)_z[v]+\mu(|x|^2)V(y)\quad \forall z\in \widetilde{M}\setminus N\;\hbox{and}\;\forall v\in T_z\widetilde{M},$$
	where $N$ is some compact subset of $M$. 
	Then $L$ is homotopically $\lsh$. In fact, for every compact set $K\subset \widetilde{M}$ there is a number $r_0=r_0(K)>0$ such that $K\subset B_0(r_0)\times M$. In the spherical polar coordinate $(r,\Theta)\in (0,\infty)\times S^{k-1}$ on $\R^k$ the standard metric can be written as $g_{\R^k}=dr^2+r^2\beta$ with the standard metric $\beta$ on the unit sphere $S^{k-1}$. Let $r_0<r_1<r_2$. Now we take a positive function $f\in C^\infty([0,\infty),\R)$ such that $f(r)=r$ on $[0,r_0]$, $f'<1$ on $[r_1,\infty) $, $0<f'\leq 1$  and $\lim_{r\to+\infty}f(r)=+\infty$. 
	
	Define the map 
	$$ \varphi: \overbar{B}_0(r_2)\setminus\{0\}\times M\longrightarrow \overbar{B}_0\big(f(r_2)\big)\times M,\quad (r,\Theta, y)\longmapsto (f(r),\Theta, y)$$
	and extend this map to a map (also denoted by $\varphi$) defined on $\overbar{B}_0(r_2)\times M$ by letting $\varphi=id$ on $\{0\}\times M$. 
	Clearly, $\varphi$ is smooth map which satisfies $\varphi^*L\leq L$ on $TM|_{\overbar{B}_0(r_2)\times M}$. So due to $\varepsilon:=r_2-f(r_2)>0$ we conclude that   $L$  is homotopically $\lsh$. 
\end{example}

\begin{example}
	If each Lagrangian $L_i:TM_i\to\R$, $i=1,\ldots, k$ is (homotopically) $\lsh$, then $L=\sum_i\pi^*_i L_i$  is (homotopically) $\lsh$ for the product Riemannian manifold $M=\prod_iM_i$, where $\pi_i$ is the projection map from $M$ to $M_i$. 
\end{example}



\end{document}